\documentclass[11pt]{article}
\usepackage{amsmath,amsthm,verbatim,amsfonts,amscd, graphicx}
\usepackage{hyperref}
\usepackage{parskip}
\usepackage{xparse}

\usepackage[utf8]{inputenc}    
\usepackage[T1]{fontenc} 
\usepackage{tikz-cd}
\usepackage{mathtools} 
\usepackage{mathrsfs}
\usepackage{enumerate} 
\usepackage{braket}
\usepackage{bigints}
\usepackage{amssymb}
\let\amslrcorner\lrcorner
\usepackage{MnSymbol}
\let\lrcorner\amslrcorner

\theoremstyle{plain}
\newtheorem{theorem}{Theorem}[section]

\newtheorem{proposition}{Proposition}[section]

\newtheorem*{claim*}{Claim}
\newtheorem*{lemma*}{Lemma}
\newtheorem*{theorem*}{Theorem}

\theoremstyle{definition}
\newtheorem{definition}{Definition}[section]

\theoremstyle{remark}
\newtheorem{remark}{Remark}

\DeclareMathOperator{\Hessian}{Hess}
\DeclareMathOperator{\distance}{dist}

\DeclareMathOperator{\divergence}{div}
\DeclareMathOperator{\Domain}{Dom}

\DeclareMathOperator{\SL}{SL}
\DeclareMathOperator{\DL}{DL}
\DeclareMathOperator{\B}{\mathcal{B}}
\DeclareMathOperator{\G}{\mathcal{G}}
\DeclareMathOperator{\finitepart}{f.p.}

\renewcommand{\Re}{\operatorname{Re}}
\renewcommand{\i}{\operatorname{\sqrt{-1}}}

\makeatletter
\newcommand{\extp}{\@ifnextchar^\@extp{\@extp^{\,}}}
\def\@extp^#1{\mathop{\bigwedge\nolimits^{\!#1}}}
\makeatother

\allowdisplaybreaks
\everymath{\displaystyle}

\NewDocumentCommand{\inn}{mo}{%
	\langle #1\rangle
	\IfValueT{#2}{^{}_{\mspace{-3mu}#2}}%
}

\NewDocumentCommand{\dinn}{mo}{%
	\llangle #1\rrangle
	\IfValueT{#2}{^{}_{\mspace{-3mu}#2}}%
}

\NewDocumentCommand{\xinn}{>{\SplitArgument{1}{,}}mo}{%
	\doinnerproduct#1
	\IfValueT{#2}{^{}_{\mspace{-3mu}#2}}%
}
\NewDocumentCommand{\doinnerproduct}{mm}{%
	\langle #1\mid #2\rangle 
}

\makeatletter
\newcommand*\bigcdot{\mathpalette\bigcdot@{1}}
\newcommand*\bigcdot@[2]{\mathbin{\vcenter{\hbox{\scalebox{#2}{$\m@th#1\bullet$}}}}}
\makeatother

\usepackage[
backend=bibtex,  
natbib=true,
url=false, 
]{biblatex}
\IfFileExists{mybibliography.bib}
{\addbibresource{mybibliography.bib}}
{}
\IfFileExists{../bib/mybibliography.bib}
{\addbibresource{../bib/mybibliography.bib}}
{}

\begin{document}
	
	\title{The $\bar{\partial}$-Neumann problem and boundary integral equations}
	
	\author{
	Bingyuan Liu
	}

	\date{\today}

	\maketitle	
	
	\begin{abstract}
		In this note, we find an equivalent boundary integral equation to the classical $\bar{\partial}$-Neumann problem. The new equation contains an equivalent regularity to the global regularity of the $\bar{\partial}$-Neumann problem. We also use the integral equations to observe a new-found rigidity property of the $\bar{\partial}$-Neumann problem on the unit ball in $\mathbb{C}^2$. 
	\end{abstract}
	
	\section{Introduction}\label{introduction}


One of the important questions of Dolbeault cohomology is to understand for which domain, the Dolbeault cohomology is trivial. This question has been solved by H\"{o}rmander in \cite{Ho65} using $L^2$ methods, which is a technique from partial differential equations. More specifically, H\"{o}rmander showed, on pseudoconvex domains, that $\bar{\partial}u=f$ ($ f\in L^2$ and $\bar{\partial}f=0$) is always solvable for $u\in\Domain(\bar{\partial})$ and $u$ satisfies the estimate $\|u\|_{L^2(\Omega)}\leq C\|f\|_{L^2(\Omega)}$. 

The $\bar{\partial}$ and its adjoint induce a natural boundary condition. This enables one to study the Poisson equation under this boundary condition. This kind of boundary value problem is called the $\bar{\partial}$-Neumann problem. We will use $-\Box$ to denote the Laplace operator with this natural boundary condition. The $\bar{\partial}$-Neumann problem is closely related to solving the classical Poisson equation, but with a non-coercive boundary condition. This non-coercive boundary condition brings some essential differences comparing with the classical Dirichlet\slash Neumann condition. One of the important difference is the global regularity, i.e., given the data smooth up to boundary, the solution is smooth up to boundary as well.

In 1965, Kohn proved the global regularity of the $\bar{\partial}$-Neumann problem when the domain is strongly pseudoconvex by observing the boundary term in the Kohn--Morrey--H\"{o}rmander formula (See Proposition \ref{MMH}). An interesting question after Kohn's result is whether the global regularity can generalized to all bounded pseudoconvex domains. In 1992, Boas--Straube gave a sufficient condition under which the global regularity holds. Mathematicians then expected an affirmative answer but Christ \cite{Ch96} gave the counterexample ``the $\beta$-worm domains'' (See Barrett's \cite{Ba92} as well) in 1996. This counterexample asserts that the global regularities do not hold for all bounded pseudoconvex domains.

There are some recent results on sufficient conditions of global regularity, see e.g., Straube \cite{St08}, Straube--Zeytuncu \cite{SZ15}, Harrington \cite{Ha08}, Harrington--Liu \cite{HL19}, Harrington--Raich \cite{HR18}, Harrington--Raich \cite{HR19}, Liu--Raich \cite{LR20} and Liu \cite{Li19}. However, the complete answer, to the cut of when the regularity holds or not, is still unknown. One difficulty is the lack of sharp tools for studying the global regularity. The aim of this note is to seek for an alternative way to express the $\bar{\partial}$-Neumann problem. Indeed, we will use the classical tools of harmonic analysis to study the solutions of the $\bar{\partial}$-Neumann problem in terms of system of integral equations involving the boundary values of the solutions and the boundary values of the conomal derivative of these solutions. This reduction to boundary brings a new viewpoint to study the global regularity. The main tool for the reduction comes from potential theory and good references about it can be found in McLean \cite{Mc00} and Hsiao--Wendland \cite{HW08}. After the theorem of the reduction (Theorem \ref{maintheorem}), we observe, in Theorem \ref{newfound}, a new-found rigidity property of the $\bar{\partial}$-Neumann problem with constant velocity in $\mathbb{B}(0,1)$. This property gives a certain relation among the constants on boundary conditions and the data of the $\bar{\partial}$-Neumann problem.

 We remark that Greiner--Stein in \cite{GS77} and Chang--Nagel--Stein in \cite{CNS92} also obtained a reduction of the $\bar{\partial}$-Neumann problem to the boundary. However, their papers used a different method like Green's functions and symbolic calculus. Consequently, their reduction is up to a smoothing operator and their formulas depend on coordinates. Our formula in this note, on the other hand, is precisely formulated (not up to a smoothing operator).

The rest of the note is organized as following. In Section \ref{pre}, we introduce background and some known classical results. Section \ref{red} will give the theorem of reduction. It also provides an equivalent condition to the global regularity of the $\bar{\partial}$-Neumann problem. The main theorem of this section is Theorem \ref{maintheorem}. At the last, in Section \ref{obs}, we exhibit the new-found rigidity property. The main theorem in this section is Theorem \ref{newfound}.

\section{Preliminary}\label{pre}
In this note, we use $(\cdot,\cdot)$ and $\inn{\cdot,\cdot}$ to denote the pointwise inner product and the $L^2$ inner product respectively. We use $\braket{\cdot|\cdot}$ to denote the action of Schwartz distributions on test functions.

Let $\Omega$ be a bounded domain with smooth boundary. Let $\delta$ be the signed distance function defined as follows:

\[\delta(z):=\begin{cases}
	-\distance(z, \partial\Omega) & z\in\Omega\\
	\distance(z, \partial\Omega) & \text{otherwise}.
\end{cases}\] 

Since $\delta$ is a signed distance function, we have the convention that $(\nabla\delta, \nabla\delta)=1$, i.e., \[\sum_i |\frac{\partial\delta}{\partial x_i}|^2+\sum_i |\frac{\partial\delta}{\partial y_i}|^2=1.\]

Traditionally, the Riemann metric gives $(\frac{\partial}{\partial x_i}, \frac{\partial}{\partial x_i})=1$. For several complex variables, this deduce that \[
	(\frac{\partial}{\partial z_i}, \frac{\partial}{\partial z_i})=(\frac{1}{2}\frac{\partial}{\partial x_i}-\frac{\i}{2}\frac{\partial}{\partial y_i}, \frac{1}{2}\frac{\partial}{\partial x_i}-\frac{\i}{2}\frac{\partial}{\partial y_i})=(\frac{1}{2}\frac{\partial}{\partial x_i}, \frac{1}{2}\frac{\partial}{\partial x_i})+(\frac{\i}{2}\frac{\partial}{\partial y_i}, \frac{\i}{2}\frac{\partial}{\partial y_i})=\frac{1}{4}+\frac{1}{4}=\frac{1}{2}
\]

For a bounded domain $\Omega$ with smooth boundary in $\mathbb{C}^2$, we let \[L=2\frac{\partial\delta}{\partial  z_2}\frac{\partial}{\partial z_1}-2\frac{\partial\delta}{\partial z_1}\frac{\partial}{\partial z_2},\] be the $(1,0)$ vector field. One can observe that $L$ is tangential to $\partial\Omega$ and \[\begin{split}
	&(L, L)=4(\frac{\partial\delta}{\partial  z_2}\frac{\partial}{\partial z_1}-\frac{\partial\delta}{\partial z_1}\frac{\partial}{\partial z_2},\frac{\partial\delta}{\partial  z_2}\frac{\partial}{\partial z_1}-\frac{\partial\delta}{\partial z_1}\frac{\partial}{\partial z_2})=2(|\frac{\partial\delta}{\partial z_2}|^2+|\frac{\partial\delta}{\partial z_1}|^2)\\=&2(|\frac{1}{2}(\frac{\partial\delta}{\partial x_2}-\i\frac{\partial\delta}{\partial y_2})|^2+|\frac{1}{2}(\frac{\partial\delta}{\partial x_1}-\i\frac{\partial\delta}{\partial y_1})|^2)=\frac{1}{2}|\nabla\delta|^2=\frac{1}{2}.
\end{split}\] Similarly, we let \[N=2\frac{\partial\delta}{\partial\bar{z}_1}\frac{\partial}{\partial z_1}+2\frac{\partial\delta}{\partial\bar{z}_2}\frac{\partial}{\partial z_2}.\] Clearly, $N$ is normal to $\partial\Omega$ and $(N,N)=\frac{1}{2}$.

We also write $\omega_{\overline{L}}=2\frac{\partial\delta}{\partial z_2}d\bar{z}_1-2\frac{\partial\delta}{\partial z_1}d\bar{z}_2$ and $\omega_{\overline{N}}=2\frac{\partial\delta}{\partial \bar{z}_1}d\bar{z}_1+2\frac{\partial\delta}{\partial \bar{z}_2}d\bar{z}_2$. One can easily conclude that $(\omega_{\overline{L}},\omega_{\overline{L}})=(\omega_{\overline{N}},\omega_{\overline{N}})=2$ and $\omega_{\overline{L}}(\overline{L})=\omega_{\overline{N}}(\overline{N})=1$. Moreover, from the convention that \[
	1=((\sqrt{2})^{-2}\omega_{\overline{L}}\wedge\omega_{\overline{N}},(\sqrt{2})^{-2}\omega_{\overline{L}}\wedge\omega_{\overline{N}}),
\] we have that \[(\omega_{\overline{L}}\wedge\omega_{\overline{N}},\omega_{\overline{L}}\wedge\omega_{\overline{N}})=4.\]

We will use $\omega_{\overline{L}}$ and $\omega_{\overline{N}}$ as the frame of $(0,1)$-forms.

Let $\bar{\partial}: \mathscr{D}_{(p,q)}'(\mathbb{C}^n)\rightarrow\mathscr{D}_{{p, q+1}}'(\mathbb{C}^n)$ be the Cauchy--Riemann operator, where $ \mathscr{D}_{(p,q)}'(\mathbb{C}^n)$ denotes the $(p, q)$-currents.

In $\mathbb{C}^2$, due to the coordinate-free definition of exterior derivative $d$: \[\begin{split}
	&\bar{\partial}u(\overline{L},\overline{N})\\=&du(\overline{L}, \overline{N})\\=&\overline{L}u(\overline{N})-\overline{N}u(\overline{L})-u([\overline{L}, \overline{N}])\\
	=&\overline{L}u_2-\overline{N}u_1-2u_1\cdot ([\overline{L},\overline{N}], \overline{L})-2u_2\cdot ([\overline{L},\overline{N}], \overline{N})
\end{split}\] 

and the fact that $([L,N], N)=([L,N], \nabla\delta)=LN\delta-NL\delta=0$,  we have that \[\bar{\partial}(u_1\omega_{\overline{L}}+u_2\omega_{\overline{N}})=\left(\overline{L}u_2-\overline{N}u_1-2u_1\cdot (L, [L,N])\right)(\omega_{\overline{L}}\wedge\omega_{\overline{N}}),\] for $u=u_1\omega_{\overline{L}}+u_2\omega_{\overline{N}}\in\mathscr{D}'_{0,1}(\mathbb{C}^2)$

Fix a bounded pseudoconvex domain $\Omega$ with smooth boundary. It is clear that $\bar{\partial}$ is unbounded on $L^2(\Omega)$. So we restrict $\bar{\partial}$ on $\Domain(\bar{\partial})\subset L^2(\Omega)$, where $\Domain(\bar{\partial})=\set{u\in L^2(\Omega): \bar{\partial}u\in L^2(\Omega)}$. The graph norm associated to $\bar{\partial}$ is defined as $\|u\|^2_{\bar{\partial}}=\|u\|_{L^2(\Omega)}^2+\|\bar{\partial}u\|_{L^2(\Omega)}^2$ so $\Domain(\bar{\partial})$ becomes a Hilbert space with the graph norm. Note that in general $\Domain(\bar{\partial})\subsetneq H^1(\Omega)$ because $\bar{\partial}$ operator only takes selected derivatives. (Here $H^1(\Omega)$ denotes the first Sobolev space on $\Omega$ and is equivalent to $W^1(\Omega)$.) Consequently, there is no obvious way to define the trace operator $\gamma$ on $\Domain(\bar{\partial})$. Because of this, we will always assume, in this note, a higher regularity of $u$, in order to define $\gamma u$. Since our goal is the global regularity, we do not lose generality with this extra assumption.

The adjoint $\bar{\partial}^*$ of $\bar{\partial}$ (for the definition of adjoint, see \cite{RS72}) is unbounded, closed and densely defined on $L^2(\Omega)$. We are going to calculate the formal adjoint of $\bar{\partial}$ using a formula of Griffiths--Harris \cite{GH94}: $\bar{\partial}^*=-*\bar{\partial}*$ on compactly supported smooth forms, where $*$ stands for the Hodge-$*$ operator. (The definition of the Hodge-$*$ operator may be different in other references. But, the $\bar{\partial}^*$ would remain the same.) 

We observe that

\[*u=*(u_1\omega_{\overline{L}})+*(u_2\omega_{\overline{N}}).\] Here \[*(u_1\omega_{\overline{L}})=\frac{1}{2}\bar u_1\omega_L\wedge\omega_N\wedge\omega_{\overline{N}}\enskip\text{and}\enskip *(u_2\omega_{\overline{N}})=-\frac{1}{2}\bar{u}_2\omega_L\wedge\omega_N\wedge\omega_{\overline{L}}.\]
So, using the invariant formula of exterior differentials, we have that \[\begin{split}
	&d*u(L, N, \overline{L}, \overline{N})=\bar{\partial}*u(L, N, \overline{L}, \overline{N})\\=&\frac{1}{2}\left(\overline{L}\bar{u}_1+2\bar{u}_1\left(([L,\overline{L}], L)+([N,\overline{L}], N)-([\overline{L},\overline{N}],\overline{N})\right)\right)\\&+\frac{1}{2}\left(\overline{N}\bar{u}_2+2\bar{u}_2\left(([L,\overline{N}], L)+([N,\overline{N}], N)+([\overline{L}, \overline{N}], \overline{L})\right)\right)\\=&\frac{1}{2}\left(\overline{L}\bar{u}_1+2\bar{u}_1\left((L, \nabla_LL)+(N,\nabla_NL)\right)\right)+\frac{1}{2}\left(\overline{N}\bar{u}_2+2\bar{u}_2\left((N, \nabla_NN)+(L, \nabla_L N)\right)\right)\\
	=&\frac{1}{2}\overline{L}\bar{u}_1+\frac{1}{2}\left(\overline{N}\bar{u}_2+2\bar{u}_2(L, [L, N])\right),
\end{split}\]
where in the last line we have used the identity $(\nabla_LL,L)=-(\nabla_LN,N)$ and $(\nabla_NN,N)=-(\nabla_NL,L)$.
Consequently, 
\[
	\bar{\partial}^* u=-*\bar{\partial}*u=-2(Lu_1+Nu_2+2u_2([L, N], L)),
\] for $u\in C^\infty_{0,1}(\overline{\Omega})\cap\Domain(\bar{\partial}^*)$.

One can further define the operator $\bar{\partial}+\bar{\partial}^*$ on $\Domain(\bar{\partial}+\bar{\partial}^*):=\Domain(\bar{\partial})\cap\Domain(\bar{\partial}^*)$. In matrix form, we may write the equation in the frame of $\omega_{\overline{L}}$ and $\omega_{\overline{N}}$:
\[(\bar{\partial}+\bar{\partial}^*)\begin{pmatrix}
	u_1\\u_2
\end{pmatrix}=\begin{pmatrix}
-2\left(Lu_1+Nu_2+2u_2\cdot ([L,N], L)\right)\\\overline{L}u_2-\overline{N}u_1-2u_1\cdot (L, [L,N])
\end{pmatrix},\] for $u=u_1\omega_{\overline{L}}+u_2\omega_{\overline{N}}\in\Domain(\bar{\partial}+\bar{\partial}^*)$, where the first component in the right matrix is the coefficient of $1$ and the second component is the coefficient of $\omega_{\overline{L}}\wedge\omega_{\overline{N}}$. 

For $u\in \mathscr{D}'_{0,0} (\mathbb{C}^2)$, we observe that $\bar{\partial}u(\overline{L})=\overline{L}u$ and $\bar{\partial}u(\overline{N})=\overline{N}u$. We obtain that $\bar{\partial}u=(\overline{L}u)\omega_{\overline{L}}+(\overline{N}u)\omega_{\overline{N}}$. By integration by parts, we find \[\bar{\partial}^*u=2(Ns)\omega_{\overline{L}}-2(Ls)\omega_{\overline{N}},\] for $u=s\omega_{\overline{L}}\wedge\omega_{\overline{N}}\in\Domain(\bar{\partial}^*)$.

In addition to the discussion for operators of $\bar{\partial}$ and $\bar{\partial}^*$, the following celebrated \textit{a priori} estimate plays a role in our note.

\begin{proposition}[Kohn--Morrey--H\"{o}rmander estimate]\label{MMH}
Let $\Omega\subset\mathbb{C}^n$ be a bounded pseudoconvex domain with smooth boundary. Let $\delta$ be the signed distance function of $\Omega$. For arbitrary $u\in C^\infty_{(0,q)}(\overline{\Omega})\cap\Domain(\bar{\partial}^*)$, we have that 
\[C\|\bar{\partial} u\|^2+C\|\bar{\partial}^* u\|^2
\geq\int_\Omega(u,u)\,dx+\int_{\partial\Omega}\Hessian_\delta(u,u)\,d\sigma.\]
\end{proposition}

Combining with a density lemma, this \textit{a priori} estimate becomes a solvability result for bounded pseudoconvex domains: \[\|u\|_{L^2(\Omega)}\leq C\|(\bar{\partial}+\bar{\partial}^*)u\|_{L^2(\Omega)},\] for all $u\in\Domain(\bar{\partial}+\bar{\partial}^*)$. Moreover, with this estimate, one can obtain, by the Cauchy--Schwartz inequality, that for a bounded pseudoconvex domain $\Omega$ with smooth boundary : \[\|u\|_{L^2(\Omega)}\leq C\|(\bar{\partial}+\bar{\partial}^*)^2u\|_{L^2(\Omega)},\] for all $u\in\Domain(\bar{\partial}+\bar{\partial}^*)$ so that $(\bar{\partial}+\bar{\partial}^*)u\in \Domain(\bar{\partial}+\bar{\partial}^*)$. Note this inequality does not guarantee a regularity because we could not improve $\|\cdot\|_{L^2(\Omega)}$ to $\|\cdot\|_{H^1(\Omega)}$ due to the boundary condition (roughly speaking, the boundary condition from $\Domain(\bar{\partial}^*)$ is not strong enough to make various integration by parts).

Motivated from the inequality $\|u\|_{L^2(\Omega)}\leq C\|(\bar{\partial}+\bar{\partial}^*)^2u\|_{L^2(\Omega)}$, one may define the operator $\Box=2(\bar{\partial}+\bar{\partial}^*)^2$ (Our definition of $\Box$ differs from the ones in Straube \cite{St10} and Chen--Shaw \cite{CS01} by a multiple of constant; in our setting, $-\Box=\Delta$ and in this way, we can easily borrow results of $\Delta$ from harmonic analysis) and with our definition, we may easily make use of the techniques from harmonic analysis (boundary integral equations). Let \[\Domain(\Box):=\set{u\in L^2(\Omega): u\in \Domain(\bar{\partial}+\bar{\partial}^*),\;\bar{\partial}u\in \Domain(\bar{\partial}^*)\;\text{and}\;\bar{\partial}^*u\in \Domain(\bar{\partial})}.\]Specifically, the following theorem is fundamental to the $\bar{\partial}$-Neumann problem. 

\begin{theorem}
	Let $\Omega$ be a bounded pseudoconvex domain with smooth boundary in $\mathbb{C}^2$. For each $f\in L^2(\Omega)$, one can find a unique $u\in\Domain(\Box)$ so that $\Box u=f$. Moreover, we have the estimate \[\|u\|_{L^2(\Omega)}\leq C\|f\|_{L^2(\Omega)}.\]
\end{theorem}

We now turn to introduce the background of boundary integral equations. For $u, v\in H^2(\Omega)$, we define the sesquilinear form: $\Phi(u, v)=\langle(\bar{\partial}+\bar{\partial}^*)u, (\bar{\partial}+\bar{\partial}^*)v\rangle$. The integration by parts of coordinate-free can be stated as follows: \[\int_{\Omega}Xf+(\divergence X)f\,dV=\int_{\Omega}Xf+2f\left((\nabla_LX, L)+(\nabla_NX, N)\right) \,dV=\int_{\partial\Omega}(X\delta) f\,d\sigma,\] for all $f\in C^\infty(\overline{\Omega})$ and all tangent vector fields $X$ of $\overline{\Omega}$. Using this integration by parts, one may obtain that \begin{equation}\label{sesqulinear}
	\Phi(u,v)=\frac{1}{2}\braket{\Box u, v}+\frac{1}{2}\braket{\B u, \gamma v}_{\partial\Omega}.
\end{equation} In the equation, $\gamma: H^1(\Omega)\rightarrow H^{\frac{1}{2}}(\partial\Omega)$ denotes the trace operator. The $\B: H^2(\Omega)\rightarrow H^{\frac{1}{2}}(\partial\Omega)$ is the conormal derivative defined as follows:
\[\B u=2\left.\begin{pmatrix}
	\overline{N}&-\overline{L}\\
	L&N
\end{pmatrix}\begin{pmatrix}
	u_1\\u_2
\end{pmatrix}\right|_{\partial\Omega}+2\left.\begin{pmatrix}
	2(L,[L,N])&0\\
	0&2([L,N],L)
\end{pmatrix}\begin{pmatrix}
	u_1\\u_2
\end{pmatrix}\right|_{\partial\Omega}.\]  Written in matrix forms, (\ref{sesqulinear}) can be represented as 
\[\begin{split}
		&4\Bigg\langle\begin{pmatrix}
		Lu_1+Nu_2+2u_2\cdot ([L,N], L)\\\overline{L}u_2-\overline{N}u_1-2u_1\cdot (L, [L,N])
	\end{pmatrix},\begin{pmatrix}
		Lu_1+Nu_2+2u_2\cdot ([L,N], L)\\\overline{L}u_2-\overline{N}u_1-2u_1\cdot (L, [L,N])
	\end{pmatrix}\Bigg\rangle\\=&
	\Bigg\langle\begin{pmatrix}
		\Box u_1\\\Box u_2
	\end{pmatrix}, \begin{pmatrix}
		v_1\\v_2
	\end{pmatrix}\Bigg\rangle+2\Bigg\langle\begin{pmatrix}
		-\overline{L}u_2+\overline{N}u_1+2u_1(L, [L,N])	\\Lu_1+Nu_2+2u_2 ([L, N], L)
	\end{pmatrix},\begin{pmatrix}
		\gamma u_1\\\gamma u_2
	\end{pmatrix}\Bigg\rangle_{\partial\Omega},
\end{split}
\] where in the equation above, $\langle\cdot, \cdot\rangle$ denotes $L^2$ component-wise inner product.

We now recall some basics from the boundary integral equations. For this topic, references are Hsiao--Wendland \cite{HW08} and McLean \cite{Mc00}. On a bounded domain with smooth boundary, partial differential equations with boundary condition may be reduced to integral equations on boundary called boundary integral equation. The foundation of this reduction is the Green's third identity. A modern reformulation of this identity is explained in the following paragraphs.

We define $\SL=\G\gamma^*$ and $\DL=\G\B^*$, where $\G$ is the fundamental solution of $\Delta u=f$ satisfies \[\Delta\G u=\G\Delta u=u\] for all $u\in\mathscr{E}'(\mathbb{C}^n)$, where $\mathscr{E}'(\mathbb{C}^n)$ denotes the set of distributions with compact support. 
\begin{theorem}[The third Green's identity]
	Let $f\in L^2(\Omega)$ and suppose $u\in H^2(\Omega)$. Define \[\tilde{f}(z)=\begin{cases}
		f & z\in\Omega\\
		0 & \text{otherwise}.
	\end{cases}\]
Suppose $\Delta u=f$ in $\Omega$. Then $u$ can be written as follows,\begin{equation}\label{repofu}
	u=\G \tilde{f}+\G\B^*(\gamma u)_{\partial\Omega}-\G\gamma^*(\B u)_{\partial\Omega},
\end{equation} in $\Omega$, where $\B^*$ and $\gamma^*$ are adjoints of the conormal derivative and the trace operator respectively.
\end{theorem} 

Putting $\gamma$ and $\B$ on the two sides of (\ref{repofu}), we obtain the following two identities: \[\gamma u=\gamma \G \tilde{f}+\gamma\DL(\gamma u)-\gamma\SL(\B u)\enskip \text{and} \enskip\B u=\B \G \tilde{f}+\B\DL(\gamma u)-\B\SL(\B u),\] in $\Omega$. We further define the following operators: \[S=\gamma\SL,\quad R=-\B\DL\enskip \text{and} \enskip T=2\gamma\DL+\mathrm{id}, \quad T^*=2\B\SL-\mathrm{id}.\]

 These imply that
 \begin{equation}\label{variationthird}
 	\begin{cases}
 		\gamma u=\gamma\G f+\gamma\SL\B u-\gamma\DL\gamma u=\gamma\G f+S\B u+\frac{1}{2}(\gamma u-T\gamma u)\\
 		\B u=\B\G f+\B\SL\B u-\B\DL\gamma u=\B\G f+\frac{1}{2}(\B u+T^*\B u)+R\gamma u.
 	\end{cases}
 \end{equation}

In the classical theory of boundary integral equations in $\mathbb{R}^n$, the operators $S$, $R$, $T$ and $T^*$ can be written as follows:
\[S\psi(x)=\int_{\partial\Omega}\G(x,y)\psi(y)\,d\sigma_y,\] \[T\psi(x)=2\lim_{\epsilon\to 0}\int_{\partial\Omega\backslash\mathbb{B}_\epsilon(z)}\left(\B_y \G(x,y)^*\right)^*\psi(y)\,d\sigma_y,\] \[T^*\psi(x)=2\lim_{\epsilon\to 0}\int_{\partial\Omega\backslash\mathbb{B}_\epsilon(z)}\B_x \G(y,x)^*\psi(y)\,d\sigma_y\] and \[R\psi(x)=-\underset{\epsilon\to 0}{\finitepart}\int_{\partial\Omega\backslash\mathbb{B}_\epsilon(x)}\B_x (\B_y \G(y,x))^*\psi(y)\,d\sigma_y,\] where the $\finitepart$ denotes the Hadamard finite part.

In the above integrals, $\G$ denotes the kernel of the fundamental solution of $\Delta=-\Box$.

\section{The reduction}\label{red}

For the following, we will discuss the $\bar{\partial}$-Neumann problem as the following form: for $f\in L^2_{0,1}(\Omega)$,
\begin{align}\label{dbarneumann}
	\begin{split}
			-\Box u=f&\quad\text{in}\quad\Omega\\
		u_2=0&\quad\text{on}\quad\partial\Omega\\
		\overline{N}u_1=-2u_1\cdot (L, [L,N])&\quad\text{on}\quad\partial\Omega
	\end{split}
\end{align}where $u=\begin{pmatrix}
	u_1\\u_2
\end{pmatrix}\in H^2(\Omega)$ represents $u=u_1\omega_{\overline{L}}+u_2\omega_{\overline{N}}$. 

We explain this reformulation as follows: Let $\gamma$ denotes the trace operator on boundary and let $(\gamma u)_i$ denote the $i$-th component of $\gamma u$ for $i=1,2$. It is obvious that, $u\in\Domain(\bar{\partial}^*)\cap  H^2_{0,1}(\overline{\Omega})$ if and only if $(\gamma u)_2=0$ (i.e., $\gamma u$ vanishes on the $\omega_{\overline{N}}$-component) on $\partial\Omega$.

Similarly, $u\in\Domain(\bar{\partial})$ and $\bar{\partial}u\in\Domain(\bar{\partial})^*$ if and only if $u_2|_{\partial\Omega}=0$ and $\overline{N}u_1=-2u_1\cdot g(L, [L,N])$ on $\partial\Omega$. In other words, one can say  $u\in\Domain(\bar{\partial})$ and $\bar{\partial}u\in\Domain(\bar{\partial})^*$ if and only if $(\gamma u)_2=0$ and $(\mathcal{B} u)_1=0$, where $(\mathcal{B} u)_i$ denotes the $i$-th component of $\mathcal{B} u$.

By the third Green's identity in matrix form, we have that,
\[u_1=\G_{11}f_1+\G_{12}f_2+\SL_{11}(\B u)_1+\SL_{12}(\B u)_2-\DL_{11}(\gamma u)_1-\DL_{12}(\gamma u)_2\] and \[u_2=\G_{21}f_1+\G_{22}f_2+\SL_{21}(\B u)_1+\SL_{22}(\B u)_2-\DL_{21}(\gamma u)_1-\DL_{22}(\gamma u)_2.\]

Applying $\gamma$ and $\B$ on both sides, in matrix form, we obtain \[\frac{1}{2}\begin{pmatrix}
	\gamma u_1\\\gamma u_2
\end{pmatrix}=\gamma\G f+\begin{pmatrix}
S_{11}(\B u)_1+S_{12}(\B u)_2\\
S_{21}(\B u)_1+S_{22}(\B u)_2
\end{pmatrix}-\frac{1}{2}\begin{pmatrix}
T_{11}\gamma u_1+T_{12}\gamma u_2\\
T_{21}\gamma u_1+T_{22}\gamma u_2
\end{pmatrix}\] and 
\[\frac{1}{2}\begin{pmatrix}
	(\B u)_1\\(\B u)_2
\end{pmatrix}=\B\G f+\frac{1}{2}\begin{pmatrix}
T^*_{11}(\B u)_1+T^*_{12}(\B u)_2\\
T^*_{21}(\B u)_1+T^*_{22}(\B u)_2
\end{pmatrix}+\begin{pmatrix}
R_{11}\gamma u_1+R_{12}\gamma u_2\\
R_{21}\gamma u_1+R_{22}\gamma u_2
\end{pmatrix}.\]

We now turn to compute the kernel of $T$, $S$ and $R$ in terms of our special boundary coordinates ($\omega_{\overline{L}}$ and $\omega_{\overline{N}}$).

Let $\mathrm{G}_0$ denote the fundamental solution of $-\Box$. Recall that if $u=s\,d\bar{z}_1+t\,d\bar{z}_2$ then \[\Box u=-\begin{pmatrix}
	\Delta&0\\0&\Delta
\end{pmatrix}\begin{pmatrix}
s\\t
\end{pmatrix}.\] It is well-known that the fundamental solution $G_0$ of $\Delta$ has the following form, \[\mathrm{G}_0(f\,d\bar{z}_1+h\,d\bar{z}_2)=\frac{1}{4}(\pi)^{-2}\int_{\mathbb{R}^4}\frac{f(w)}{|z-w|^2}\,d\bar{z}_1+\frac{h(w)}{|z-w|^2}\,d\bar{z}_2\,d\sigma_w=\mathrm{G}_0f\,d\bar{z}_1+\mathrm{G}_0h\,d\bar{z}_2.\]


The formula for $G_0$ is based on the classical coordinates $d\bar{z}_1$ and $d\bar{z}_2$. In order to find the fundamental solution $\G$ based on the coordinates $\omega_{\overline{L}}$ and $\omega_{\overline{N}}$, we do coordinate changes as follows:
\begin{align*}
	&\G(u_1\omega_{\overline{L}}+u_2\omega_{\overline{N}})\\
	=&\frac{1}{2}\mathrm{G}_0(u_1\omega_{\overline{L}}+u_2\omega_{\overline{N}},\,d\bar{z}_1)\,d\bar{z}_1+\frac{1}{2}\mathrm{G}_0(u_1\omega_{\overline{L}}+u_2\omega_{\overline{N}},\,d\bar{z}_2)\,d\bar{z}_2\\
	=&\frac{1}{4}(\mathrm{G}_0(u_1\omega_{\overline{L}}+u_2\omega_{\overline{N}},\,d\bar{z}_1)\,d\bar{z}_1+\mathrm{G}_0(u_1\omega_{\overline{L}}+u_2\omega_{\overline{N}},\,d\bar{z}_2)\,d\bar{z}_2,\omega_{\overline{L}})\omega_{\overline{L}}\\&+\frac{1}{4}(\mathrm{G}_0(u_1\omega_{\overline{L}}+u_2\omega_{\overline{N}},\,d\bar{z}_1)\,d\bar{z}_1+\mathrm{G}_0(u_1\omega_{\overline{L}}+u_2\omega_{\overline{N}},\,d\bar{z}_2)\,d\bar{z}_2,\omega_{\overline{N}})\omega_{\overline{N}}\\
	=&4\begin{pmatrix}
		\frac{\partial\delta}{\partial\bar{z}_2}&-\frac{\partial\delta}{\partial\bar{z}_1}\\\frac{\partial\delta}{\partial z_1}&\frac{\partial\delta}{\partial z_2}
	\end{pmatrix}\begin{pmatrix}
	\mathrm{G}_0&0\\0&\mathrm{G}_0
\end{pmatrix}\begin{pmatrix}
\frac{\partial\delta}{\partial z_2}&\frac{\partial\delta}{\partial\bar{z}_1}\\-\frac{\partial\delta}{\partial z_1}&\frac{\partial\delta}{\partial\bar{z}_2}
\end{pmatrix}\begin{pmatrix}
u_1\\u_2
\end{pmatrix}\\
=&4\begin{pmatrix}
	\frac{\partial\delta}{\partial\bar{z}_2}G_0(\frac{\partial\delta}{\partial z_2}u_1+\frac{\partial\delta}{\partial\bar{z}_1}u_2)-\frac{\partial\delta}{\partial\bar{z}_1}G_0(-\frac{\partial\delta}{\partial z_1}u_1+\frac{\partial\delta}{\partial\bar{z}_2}u_2)\\
	\frac{\partial\delta}{\partial z_1}G_0(\frac{\partial\delta}{\partial z_2}u_1+\frac{\partial\delta}{\partial\bar{z}_1}u_2)+\frac{\partial\delta}{\partial z_2}G_0(-\frac{\partial\delta}{\partial z_1}u_1+\frac{\partial\delta}{\partial\bar{z}_2}u_2)
\end{pmatrix}\\
=&\begin{pmatrix}
	\G_{11}&\G_{12}\\
	\G_{21}&\G_{22}
\end{pmatrix}\begin{pmatrix}
u_1\\u_2
\end{pmatrix}.
\end{align*}

Here,  \begin{equation}\label{forG}
	\begin{split}
		\G_{11}u_1&=(\pi)^{-2}\int\frac{\frac{\partial\delta}{\partial\bar{z}_2}(z)\frac{\partial\delta}{\partial z_2}(w)+\frac{\partial\delta}{\partial\bar{z}_1}(z)\frac{\partial\delta}{\partial z_1}(w)}{|z-w|^2}u_1(w)\,d\sigma_w\\
		\G_{12}u_2&=(\pi)^{-2}\int\frac{\frac{\partial\delta}{\partial\bar{z}_2}(z)\frac{\partial\delta}{\partial \bar{z}_1}(w)-\frac{\partial\delta}{\partial\bar{z}_1}(z)\frac{\partial\delta}{\partial\bar{z}_2}(w)}{|z-w|^2}u_2(w)\,d\sigma_w\\
		\G_{21}u_1&=(\pi)^{-2}\int\frac{\frac{\partial\delta}{\partial z_1}(z)\frac{\partial\delta}{\partial z_2}(w)-\frac{\partial\delta}{\partial z_2}(z)\frac{\partial\delta}{\partial z_1}(w)}{|z-w|^2}u_1(w)\,d\sigma_w\\
		\G_{22}u_2&=(\pi)^{-2}\int\frac{\frac{\partial\delta}{\partial z_1}(z)\frac{\partial\delta}{\partial\bar{z}_1}(w)+\frac{\partial\delta}{\partial z_2}(z)\frac{\partial\delta}{\partial\bar{z}_2}(w)}{|z-w|^2}u_2(w)\,d\sigma_w
	\end{split}
\end{equation}


Since we need to compute kernels which have $z=(z_1, z_2), w=(w_1, w_2)$ variables, it is inevitable to have product of functions in $z$ and functions in $w$. For this, we define the following notation \[\braket{N_z\cdot N_w}=\braket{2\frac{\partial\delta}{\partial\bar{z}_1}\frac{\partial}{\partial z_1}+2\frac{\partial\delta}{\partial\bar{z}_2}\frac{\partial}{\partial z_2}\cdot 2\frac{\partial\delta}{\partial\bar{w}_1}\frac{\partial}{\partial w_1}+2\frac{\partial\delta}{\partial\bar{w}_2}\frac{\partial}{\partial w_2}}=2\left(\frac{\partial\delta}{\partial\bar{z}_2}(z)\frac{\partial\delta}{\partial z_2}(w)+\frac{\partial\delta}{\partial\bar{z}_1}(z)\frac{\partial\delta}{\partial z_1}(w)\right).\] 

We also abuse the notation and denote $2\frac{\partial\delta}{\partial \bar{z}_1}(\bar{z}_1-\bar{w}_1)+2\frac{\partial\delta}{\partial \bar{z}_2}(\bar{z}_2-\bar{w}_2)$ by $\braket{N_z\cdot (z-w)}$. One can check that \[|\braket{N_z\cdot (z-w)}|^2+|\braket{L_z\cdot (z-w)}|^2=2|z-w|^2.\]

Similarly, we can define $\braket{N_z\cdot L_w}$, $\braket{L_z\cdot L_w}$, $\braket{L_z\cdot (z-w)}$.

Now, we can simplify (\ref{forG}) to \[\G_{11}u_1=\frac{1}{2}(\pi)^{-2}\int\frac{\braket{L_w\cdot L_z}}{|z-w|^2}u_1(w)\,d\sigma_w=\frac{1}{2}(\pi)^{-2}\int\frac{\braket{N_z\cdot N_w}}{|z-w|^2}u_1(w)\,d\sigma_w\]
\[\G_{12}u_2=\frac{1}{2}(\pi)^{-2}\int\frac{\braket{N_w\cdot L_z}}{|z-w|^2}u_2(w)\,d\sigma_w=-\frac{1}{2}(\pi)^{-2}\int\frac{\braket{N_z\cdot L_w}}{|z-w|^2}u_2(w)\,d\sigma_w\]
\[\G_{21}u_1=\frac{1}{2}(\pi)^{-2}\int\frac{\braket{L_w\cdot N_z}}{|z-w|^2}u_1(w)\,d\sigma_w=-\frac{1}{2}(\pi)^{-2}\int\frac{\braket{L_z\cdot N_w}}{|z-w|^2}u_1(w)\,d\sigma_w\]
\[\G_{22}u_2=\frac{1}{2}(\pi)^{-2}\int\frac{\braket{N_w\cdot N_z}}{|z-w|^2}u_2(w)\,d\sigma_w=\frac{1}{2}(\pi)^{-2}\int\frac{\braket{L_z\cdot L_w}}{|z-w|^2}u_2(w)\,d\sigma_w\]

Applying the conormal derivative, we get the kernel of $T^*$ (For the following, we abuse the notation $\G$ for both the fundamental solution and the kernel of the fundamental solution.):

\begin{align*}
	\B_z\G(z,w)
	=(\pi)^{-2}\left(\begin{pmatrix}
		\overline{N}&-\overline{L}\\
		L&N
	\end{pmatrix}+2\begin{pmatrix}
		(L,[L,N])&0\\
		0&([L,N],L)
	\end{pmatrix}\right)\begin{pmatrix}
	\frac{\braket{L_w\cdot L_z}}{|z-w|^2}&\frac{\braket{N_w\cdot L_z}}{|z-w|^2}\\\frac{\braket{L_w\cdot N_z}}{|z-w|^2}&\frac{\braket{N_w\cdot N_z}}{|z-w|^2}
\end{pmatrix}
\end{align*}


To simplify the matrix above, we use  \[\frac{\partial}{\partial z_j}\frac{1}{|z-w|^2}=-\frac{\bar{z}_j-\bar{w}_j}{|z-w|^4}\enskip \text{and} \enskip\frac{\partial}{\partial w_j}\frac{1}{|z-w|^2}=-\frac{\bar{w}_j-\bar{z}_j}{|z-w|^4}.\]

Consequently, we can simplify the kernel of $T^*$ as follows. We first formulate $\B_z\G(z,w)$ with two equivalent forms: \[\begin{split}
	&\B_z\G(z,w)\\
=&(\pi)^{-2}\left(\begin{pmatrix}
	\overline{N}&-\overline{L}\\
	L&N
\end{pmatrix}+2\begin{pmatrix}
	(L,[L,N])&0\\
	0&([L,N],L)
\end{pmatrix}\right)\begin{pmatrix}
	\frac{\braket{L_w\cdot L_z}}{|z-w|^2}&\frac{\braket{N_w\cdot L_z}}{|z-w|^2}\\\frac{\braket{L_w\cdot N_z}}{|z-w|^2}&\frac{\braket{N_w\cdot N_z}}{|z-w|^2}
\end{pmatrix}\\
=&(\pi)^{-2}\left(\begin{pmatrix}
	\overline{N}&-\overline{L}\\
	L&N
\end{pmatrix}+2\begin{pmatrix}
	(L,[L,N])&0\\
	0&([L,N],L)
\end{pmatrix}\right)\begin{pmatrix}
	\frac{\braket{N_z\cdot N_w}}{|z-w|^2}&-\frac{\braket{N_z\cdot L_w}}{|z-w|^2}\\-\frac{\braket{L_z\cdot N_w}}{|z-w|^2}&\frac{\braket{L_z\cdot L_w}}{|z-w|^2}
\end{pmatrix}.
\end{split}\] Indeed, \[\begin{pmatrix}
\frac{\braket{L_w\cdot L_z}}{|z-w|^2}&\frac{\braket{N_w\cdot L_z}}{|z-w|^2}\\\frac{\braket{L_w\cdot N_z}}{|z-w|^2}&\frac{\braket{N_w\cdot N_z}}{|z-w|^2}
\end{pmatrix}=\begin{pmatrix}
\frac{\braket{N_z\cdot N_w}}{|z-w|^2}&-\frac{\braket{N_z\cdot L_w}}{|z-w|^2}\\-\frac{\braket{L_z\cdot N_w}}{|z-w|^2}&\frac{\braket{L_z\cdot L_w}}{|z-w|^2}
\end{pmatrix}.\]

We will then simplify $\B_z\G(z,w)$ using terms from the above two equivalent forms whichever is more helpful. For example, in the first form, \[\overline{N}_z\braket{L_w\cdot L_z}-\overline{L}_z\braket{L_w\cdot N_z}=\braket{L_w\cdot [N, L]_z}\] while in the second form $\overline{N}_z\braket{N_z\cdot N_w}+\overline{L}_z\braket{L_z\cdot N_w}$ does not seem to be easily simplified.
We simplify that
\begin{align*}
	&\B_z\G(z,w)\\
=&(\pi)^{-2}\Bigg(\begin{pmatrix}
	\frac{\braket{L_w\cdot [N,L]_z}}{|z-w|^2}&\frac{\braket{N_w\cdot [N,L]_z}}{|z-w|^2}\\
	-\frac{\braket{[N,L]_z\cdot N_w}}{|z-w|^2}&\frac{\braket{[N,L]_z\cdot L_w}}{|z-w|^2}
\end{pmatrix}\\&-\begin{pmatrix}
\frac{\braket{N_z\cdot N_w}\braket{(z-w), N}_z+\braket{L_z\cdot N_w}\braket{(z-w), L}_z}{|z-w|^4}&\frac{-\braket{N_z\cdot L_w}\braket{(z-w), N}_z-\braket{L_z\cdot L_w}\braket{(z-w),L}_z}{|z-w|^4}\\
\frac{\braket{L_w\cdot N_z}\braket{N, (z-w)}_z+\braket{L_w\cdot L_z}\braket{L, (z-w)}_z}{|z-w|^4}&\frac{\braket{N_w\cdot N_z}\braket{N, (z-w)}_z+\braket{N_w\cdot L_z}\braket{L, (z-w)}_z}{|z-w|^4}
\end{pmatrix}\\&+\begin{pmatrix}
	\frac{\braket{L_w\cdot 2([L,N]_z, L_z)L_z}}{|z-w|^2}&\frac{\braket{N_w\cdot 2([L,N]_z, L_z)L_z}}{|z-w|^2}\\
	-\frac{\braket{2([L_z,N_z],L_z)L_z\cdot N_w}}{|z-w|^2}&\frac{\braket{2([L_z,N_z],L_z)L_z\cdot L_w}}{|z-w|^2}
\end{pmatrix}\Bigg)\\
=&-\frac{1}{2}(\pi)^{-2}\begin{pmatrix}
	\frac{\braket{(z-w)\cdot N_w}}{|z-w|^4}&-\frac{\braket{(z-w)\cdot L_w}}{|z-w|^4}\\
	\frac{\braket{L_w\cdot (z-w)}}{|z-w|^4}&\frac{\braket{N_w\cdot (z-w)}}{|z-w|^4}
\end{pmatrix}.
\end{align*} The last equation is due to the simple observation $([L,N], N)=([L,N], \nabla\delta)=[L,N]\delta=0$.

Furthermore, by switching $z$ and $w$, \[\B_w\G(w,z)=-\frac{1}{2}(\pi)^{-2}\begin{pmatrix}
\frac{\braket{(w-z)\cdot N_z}}{|z-w|^4}&-\frac{\braket{(w-z)\cdot L_z}}{|z-w|^4}\\
\frac{\braket{L_z\cdot (w-z)}}{|z-w|^4}&\frac{\braket{N_z\cdot (w-z)}}{|z-w|^4}
\end{pmatrix}.\]

The kernel of $T$ can be easily obtained as follows: \[(\B_w\G(w,z))^*=-\frac{1}{2}(\pi)^{-2}\begin{pmatrix}
	\frac{\braket{N_z\cdot (w-z)}}{|z-w|^4}&\frac{\braket{(w-z)\cdot L_z}}{|z-w|^4}\\
	-\frac{\braket{L_z\cdot (w-z)}}{|z-w|^4}&\frac{\braket{(w-z)\cdot N_z}}{|z-w|^4}
\end{pmatrix},\] where $*$ denotes the transpose and conjugation.

To obtain the kernel of $R$, we define $(w-z)^\perp=\begin{pmatrix}
	w_1-z_1\\w_2-z_2
\end{pmatrix}^\perp=\begin{pmatrix}
	w_2-z_2\\ z_1-w_1
\end{pmatrix}$. In this way, we may get another expression of $(\B_w\G(w,z))^*$ as follows \[(\B_w\G(w,z))^*=-\frac{1}{2}(\pi)^{-2}\begin{pmatrix}
\frac{\braket{(\bar{w}-\bar{z})^\perp\cdot L_z}}{|z-w|^4}&-\frac{\braket{N_z\cdot (\bar{w}-\bar{z})^\perp}}{|z-w|^4}\\
\frac{\braket{(\bar{w}-\bar{z})^\perp\cdot N_z}}{|z-w|^4}&\frac{\braket{L_z\cdot (\bar{w}-\bar{z})^\perp}}{|z-w|^4}
\end{pmatrix}.\]

Observe that \[N_z(|z-w|^2)^{-2}=-2(|z-w|^2)^{-3}\left(2\frac{\partial\delta}{\partial \bar{z}_1}(\bar{z}_1-\bar{w}_1)+2\frac{\partial\delta}{\partial \bar{z}_2}(\bar{z}_2-\bar{w}_2)\right)=-2(|z-w|^2)^{-3}\braket{N_z\cdot z-w}\] and \[L_z(|z-w|^2)^{-2}=-2(|z-w|^2)^{-3}\braket{L_z\cdot z-w}.\]
We now ready to compute the kernel of $R=R^*$:
\begin{align*}
	-\B_z(\B_w\G(w,z))^*
	=&(\pi)^{-2}\Bigg(\begin{pmatrix}
		\overline{N}_z&-\overline{L}_z\\
		L_z&N_z
	\end{pmatrix}\begin{pmatrix}
	\frac{\braket{(\bar{w}-\bar{z})^\perp\cdot L_z}}{|z-w|^4}&-\frac{\braket{N_z\cdot (\bar{w}-\bar{z})^\perp}}{|z-w|^4}\\
	\frac{\braket{(\bar{w}-\bar{z})^\perp\cdot N_z}}{|z-w|^4}&\frac{\braket{L_z\cdot (\bar{w}-\bar{z})^\perp}}{|z-w|^4}
\end{pmatrix}\\&+\begin{pmatrix}
		2(L,[L,N])_z&0\\
		0&2([L,N],L)_z
	\end{pmatrix}\begin{pmatrix}
	\frac{\braket{(\bar{w}-\bar{z})^\perp\cdot L_z}}{|z-w|^4}&-\frac{\braket{N_z\cdot (\bar{w}-\bar{z})^\perp}}{|z-w|^4}\\
	\frac{\braket{(\bar{w}-\bar{z})^\perp\cdot N_z}}{|z-w|^4}&\frac{\braket{L_z\cdot (\bar{w}-\bar{z})^\perp}}{|z-w|^4}
\end{pmatrix}\Bigg)\\
=&(\pi)^{-2}\Bigg(\begin{pmatrix}
	-\frac{2}{|z-w|^4}&0\\
	0&-\frac{2}{|z-w|^4}
\end{pmatrix}+\begin{pmatrix}
	-\frac{2}{|z-w|^4}&0\\0&-\frac{2}{|z-w|^4}
\end{pmatrix}\Bigg)\\
=&-4(\pi)^{-2}\begin{pmatrix}
	\frac{1}{|z-w|^4}&0\\
	0&\frac{1}{|z-w|^4}
\end{pmatrix}
\end{align*}

We have obtained all kernels.

\begin{proposition}\label{calc}
	The operators $S: H^{s-\frac{1}{2}}(\partial\Omega)\rightarrow H^{s+\frac{1}{2}}(\partial\Omega)$, $T: H^{s+\frac{1}{2}}(\partial\Omega)\rightarrow H^{s+\frac{1}{2}}(\partial\Omega)$ and $R:H^{s+\frac{1}{2}}(\partial\Omega)\rightarrow H^{s-\frac{1}{2}}(\partial\Omega)$ for all $s\geq 0$ have the following explicit formulae: 
	\[S\begin{pmatrix}
		\psi_1\\\psi_2
	\end{pmatrix}(z)=\frac{1}{2}(\pi)^{-2}\bigint_{\partial\Omega}\begin{pmatrix}
		\frac{\braket{L_w\cdot L_z}}{|z-w|^2}&\frac{\braket{N_w\cdot L_z}}{|z-w|^2}\\
		\frac{\braket{L_w\cdot N_z}}{|z-w|^2}&\frac{\braket{N_w\cdot N_z}}{|z-w|^2}
	\end{pmatrix}\begin{pmatrix}
		\psi_1(w)\\\psi_2(w)
	\end{pmatrix}\,d\sigma_w\]
	
	\[T\begin{pmatrix}
		\psi_1\\\psi_2
	\end{pmatrix}(z)=-(\pi)^{-2}\lim_{\epsilon\to 0}\bigint_{\partial\Omega\backslash\mathbb{B}_\epsilon(z)}\begin{pmatrix}
		\frac{\braket{N_z\cdot (w-z)}}{|z-w|^4}&\frac{\braket{(w-z)\cdot L_z}}{|z-w|^4}\\
		-\frac{\braket{L_z\cdot (w-z)}}{|z-w|^4}&\frac{\braket{(w-z)\cdot N_z}}{|z-w|^4}
	\end{pmatrix}\begin{pmatrix}
		\psi_1(w)\\\psi_2(w)
	\end{pmatrix}\,d\sigma_w\]
	and 
	\[R\begin{pmatrix}
		\psi_1\\\psi_2
	\end{pmatrix}(z)=4(\pi)^{-2}\underset{\epsilon\to 0}{\finitepart}\bigint_{\partial\Omega\backslash\mathbb{B}_\epsilon(z)}\begin{pmatrix}
		\frac{1}{|z-w|^4}&0\\
		0&\frac{1}{|z-w|^4}
	\end{pmatrix}\begin{pmatrix}
		\psi_1(w)\\\psi_2(w)
	\end{pmatrix}\,d\sigma_w\]
\end{proposition}

Our main theorem of the current section is as follows. 

\begin{theorem}\label{maintheorem}
	Let $\Omega$ be a bounded pseudoconvex domain with smooth boudnary in $\mathbb{C}^2$. Let $f\in L^2(\Omega)$. 
	\begin{enumerate}
		\item\label{1} If $u\in H^2(\Omega)$ is a solution of the $\bar{\partial}$-Neumann problem (\ref{dbarneumann}), then $\psi=\begin{pmatrix}
			\psi_1\\\psi_2
		\end{pmatrix}=\gamma u\in H^{\frac{3}{2}}(\partial\Omega)$ and $\phi=\begin{pmatrix}
			\phi_1\\\phi_2
		\end{pmatrix}=\B u\in H^{\frac{1}{2}}(\partial\Omega)$ satisfies the boundary integral equation:
		\begin{equation}\label{boundaryintegral}
			\begin{cases}
				\frac{1}{2}\begin{pmatrix}
					\psi_1\\0
				\end{pmatrix}&=\gamma\G f+
				S\phi-\frac{1}{2}T\psi\\
				\frac{1}{2}\begin{pmatrix}
					0\\\phi_2
				\end{pmatrix}&=\B\G f+\frac{1}{2}T^*\phi+R\psi.
			\end{cases}		
		\end{equation}
		\item\label{2} On the contrary, if $\psi$ and $\phi$ satisfy the (\ref{boundaryintegral}), then $u=\G f+\SL\phi-\DL\psi$ is the unique solution of the $\bar{\partial}$-Neumann problem.
	\end{enumerate}
	
	Moreover, the global regularity holds for the $\bar{\partial}$-Neumann problem if and only if the solutions of $\psi$ and $\phi$ are smooth for $f\in C^\infty(\overline{\Omega})$.
\end{theorem}

\begin{proof}[Proof of Theorem \ref{maintheorem}]
	We prove \ref{1} first. Let $u\in H^2(\Omega)$ be a solution of the $\bar{\partial}$-Neumann problem. Plug the boundary condition $(\gamma u)_2=(\B u)_1=0$ into (\ref{variationthird}), and we have that
	\[\frac{1}{2}\begin{pmatrix}
		\gamma u_1\\0
	\end{pmatrix}=\gamma\G f+S\B u-\frac{1}{2}T\gamma u\enskip \text{and} 
	\enskip\frac{1}{2}\begin{pmatrix}
		0\\(\B u)_2
	\end{pmatrix}=\B\G f+\frac{1}{2}T^*\B u+R\gamma u,\] which is (\ref{boundaryintegral}). 
	
	For \ref{2}, we just need to verify that $u=\G f+\SL\phi-\DL\psi$ is one of the solutions of $\bar{\partial}$-Neumann problem. Then by the uniqueness of solutions, we may conclude the proof of \ref{2}. Clearly, $\Delta\SL\phi=0$ in $\Omega$ because for any $s\in C^\infty_c(\Omega)$, we have that \[\braket{\Delta\SL\psi|s}_\Omega=\braket{\G\gamma^*\psi|-\Delta s}_\Omega=-\braket{\psi|\gamma\G\Delta s}_{\partial\Omega}=-\braket{\psi|\gamma s}_{\partial\Omega}=0.\] Similarly $\Delta\DL\psi=0$ in $\Omega$. Consequently, $-\Box u=f$ in $\Omega$. 
	
	We now verify the boundary condition. For this aim we take \[\gamma u=\gamma\G f+\SL \phi-\gamma\DL\psi=\gamma\G f+S\B u+\frac{1}{2}(\gamma u-T\gamma u).\] By (\ref{boundaryintegral}), we see $(\gamma u)_2=0$. Similarly, we also get $(\B u)_1=0$. Thus, \ref{2} follows.
	
	The equivalence of regularities follows immediately by the following regularity fact: for all $s\in\mathbb{R}$, $\SL: H^{s-\frac{1}{2}}(\partial\Omega\cap G_1)\rightarrow H^{s+1}(G_2\cap\Omega)$ and $\DL: H^{s+\frac{1}{2}}(\partial\Omega\cap G_1)\rightarrow H^{s+1}(G_2\cap\Omega)$ for two open sets $\overline{G}_1\Subset G_2$.
\end{proof}

\begin{remark}
	Indeed, by observing the linearity ($S_{ij}(0)=T_{ij}(0)=R_{ij}(0)=0$) of $S$, $T$, $T^*$ and $R$, we can simplify \ref{1} a bit as follows. Under the assumption of Theorem \ref{maintheorem}, if $u\in H^2(\Omega)$ is the solution of $\bar{\partial}$-Neumann problem, then $\psi=\begin{pmatrix}
		\psi_1\\\psi_2
	\end{pmatrix}=\gamma u\in H^{\frac{3}{2}}(\partial\Omega)$ and $\phi=\begin{pmatrix}
		\phi_1\\\phi_2
	\end{pmatrix}=\B u\in H^{\frac{1}{2}}(\partial\Omega)$ satisfies the boundary integral equation:
	\begin{equation*}
		\begin{cases}
			\frac{1}{2}\begin{pmatrix}
				\psi_1\\0
			\end{pmatrix}&=\gamma\G f+
			\begin{pmatrix}
				S_{12}\phi_2\\
				S_{22}\phi_2
			\end{pmatrix}-\frac{1}{2}\begin{pmatrix}
				T_{11}\psi_1\\
				T_{21}\psi_1
			\end{pmatrix}\\
			\frac{1}{2}\begin{pmatrix}
				0\\\phi_2
			\end{pmatrix}&=\B\G f+\frac{1}{2}\begin{pmatrix}
			T^*_{12}\phi_2\\
			T^*_{22}\phi_2
		\end{pmatrix}+\begin{pmatrix}
		R_{11}\psi_1\\
		R_{21}\psi_1\end{pmatrix}.
		\end{cases}
	\end{equation*}

\end{remark}

\section{The $\bar{\partial}$-Neumann problem with constant velocity on the unit ball}\label{obs}

In this section we are going to discuss the boundary values of the solution $u\in H^2(\Omega)$ of the $\bar{\partial}$-Neumann problem on $\Omega=\mathbb{B}(0,1)\subset\mathbb{C}^2$. As discussed before, there are in total $4$ boundary values $(\gamma u)_1$, $(\gamma u)_2$, $(\B u)_1$ and $(\B u)_1$. The boundary conditions of the $\bar{\partial}$-Neumann problem have already fixed two of them: $(\gamma u)_2=(\B u)_1=0$. But $(\gamma u)_1$ and $(\B u)_2$ are still free. In this section, we will study the relation among $(\B u)_2=:\phi_2$, $(\gamma u)_1=:\psi_1$ and the data $f$ in (\ref{dbarneumann}). To the best of author's knowledge, this kind of problems has not been studied yet. Indeed, Theorem \ref{maintheorem} provides a way to study this problem. We will focus on the $\bar{\partial}$-Neumann problem on the unit ball with constant velocity as follows.

We first define the $\bar{\partial}$-Neumann problem with constant boundary velocity.
\begin{definition}
	Let $\Omega\subset\mathbb{C}^n$ be a bounded pseudoconvex domain with smooth boundary. Given $f\in L^2(\Omega)$, let $u\in H^2(\Omega)$ be a solution of the $\bar{\partial}$-Neumann problem. We will call (\ref{dbarneumann}) the $\bar{\partial}$-Neumann problem with constant boundary velocity if $\gamma u$ and $\B u$ are both constant vectors on $\partial\Omega$.
\end{definition}

\begin{remark}
	In $\mathbb{C}^2$, the boundary condition of the $\bar{\partial}$-Neumann problem gives $(\gamma u)_2=(\B u)_1=0$. Consequently, to study the problem with constant boundary velocity, we can assume $\psi_1=(\gamma u)_1=b\in\mathbb{C}$ and $\phi_2=(\B u)_2=a\in\mathbb{C}$.
\end{remark}

We start observation by computing kernels on the unit ball. On the unit ball $\Omega=\mathbb{B}(0,1)$, the signed distance function is $\delta(z_1, z_2)=\sqrt{|z_1|^2+|z_2|^2}-1$. So \[L_z=\frac{\bar{z}_2}{\sqrt{|z_1|^2+|z_2|^2}}\frac{\partial}{\partial z_1}-\frac{\bar{z}_1}{\sqrt{|z_1|^2+|z_2|^2}}\frac{\partial}{\partial z_2}\enskip\text{and}\enskip N_z=\frac{z_1}{\sqrt{|z_1|^2+|z_2|^2}}\frac{\partial}{\partial z_1}+\frac{z_2}{\sqrt{|z_1|^2+|z_2|^2}}\frac{\partial}{\partial z_2}.\] 

Observe that  on $\partial\Omega=\partial\mathbb{B}(0,1)$, \[\inn{L_w\cdot L_z}=\frac{1}{2}(\bar{w}_2z_2+\bar{w}_1z_1)\enskip\text{and}\enskip \inn{N_w\cdot L_z}=\frac{1}{2}(w_1z_2-w_2z_1)\enskip\text{and}\enskip \inn{N_w\cdot N_z}=\frac{1}{2}(w_1\bar{z}_1+w_2\bar{z}_2).\]
Thus, by Proposition \ref{calc}, \[S\begin{pmatrix}
	\phi_1\\\phi_2
\end{pmatrix}(z)=\frac{1}{4\pi^2}\int_{\partial{\Omega}}|z-w|^{-2}\begin{pmatrix}
(\bar{w}_2z_2+\bar{w}_1z_1)\phi_1(w)+(w_1z_2-w_2z_1)\phi_2(w)\\(\bar{w}_2\bar{z}_1-\bar{w}_1\bar{z}_2)\phi_1(w)+(w_1\bar{z}_1+w_2\bar{z}_2)\phi_2(w)
\end{pmatrix}\,d\sigma_w\]
\[\begin{split}
&T\begin{pmatrix}
	\psi_1\\\psi_2
\end{pmatrix}(z)\\=&-(\pi)^{-2}\lim_{\epsilon\to 0}\int_{\partial\Omega\backslash\mathbb{B}_\epsilon(z)}|z-w|^{-4}\begin{pmatrix}
(-1+z_1\bar{w}_1+z_2\bar{w}_2)\psi_1(w)+(z_2w_1-z_1w_2)\psi_2(w)\\
	-(\bar{z}_2\bar{w}_1-\bar{z}_1\bar{w}_2)\psi_1(w)+(-1+\bar{z}_1w_1+\bar{z}_2w_2)\psi_2(w)
\end{pmatrix}\,d\sigma_w\\
=&-(\pi)^{-2}\lim_{\epsilon\to 0}\int_{\partial\Omega\backslash\mathbb{B}_\epsilon(z)}|z-w|^{-4}\begin{pmatrix}
	(-\frac{|z-w|^2}{2}+\frac{z_1\bar{w}_1+z_2\bar{w}_2-\bar{z}_1w_1-\bar{z}_2w_2}{2})\psi_1(w)+(z_2w_1-z_1w_2)\psi_2(w)\\
	-(\bar{z}_2\bar{w}_1-\bar{z}_1\bar{w}_2)\psi_1(w)+(-\frac{|z-w|^2}{2}+\frac{\bar{z}_1w_1+\bar{z}_2w_2-z_1\bar{w}_1-z_2\bar{w}_2}{2})\psi_2(w)
\end{pmatrix}\,d\sigma_w.
\end{split}\] In the last equation, we have used the identity \[|z-w|^2=|z|^2+|w|^2-2\Re z_1\bar{w}_1-2\Re z_2\bar{w}_2=2-2\Re z_1\bar{w}_1-2\Re z_2\bar{w}_2\] on $\partial\Omega$.

It is not hard to compute $T^*\phi$
\[T^*\phi=-(\pi)^{-2}\lim_{\epsilon\to 0}\int_{\partial\Omega\backslash\mathbb{B}_\epsilon(z)}|z-w|^{-4}\begin{pmatrix}
	(-1+\bar{w}_1z_1+\bar{w}_2z_2)\phi_1(w)+(-w_2z_1+w_1z_2)\phi_2(w)\\
	(\bar{w}_2\bar{z}_1-\bar{w}_1\bar{z}_2)\phi_1(w)+(-1+w_1\bar{z}_1+w_2\bar{z}_2)\phi_2(w)
\end{pmatrix}\,d\sigma_w\]

and $R\psi$
	\[R\begin{pmatrix}
	\psi_1\\\psi_2
\end{pmatrix}(z)=4(\pi)^{-2}\underset{\epsilon\to 0}{\finitepart}\bigint_{\partial\Omega\backslash\mathbb{B}_\epsilon(z)}\begin{pmatrix}
	\frac{1}{|z-w|^4}&0\\
	0&\frac{1}{|z-w|^4}
\end{pmatrix}\begin{pmatrix}
	\psi_1(w)\\\psi_2(w)
\end{pmatrix}\,d\sigma_w.\]

Now, let $\psi_2=\phi_1=0$ and let $\phi_2=a, \psi_1=b$ be constants. 

Clearly, \[S\begin{pmatrix}
	0\\\phi_2
\end{pmatrix}(z)=\frac{a}{4\pi^2}\int_{\partial{\Omega}}|z-w|^{-2}\begin{pmatrix}
w_1z_2-w_2z_1\\w_1\bar{z}_1+w_2\bar{z}_2
\end{pmatrix}\,d\sigma_w\]

We also have that
\[
	T\begin{pmatrix}
		\psi_1\\0
	\end{pmatrix}(z)=-b(\pi)^{-2}\lim_{\epsilon\to 0}\int_{\partial\Omega\backslash\mathbb{B}_\epsilon(z)}|z-w|^{-4}\begin{pmatrix}
		-1+z_1\bar{w}_1+z_2\bar{w}_2\\
		-(\bar{z}_2\bar{w}_1-\bar{z}_1\bar{w}_2)
	\end{pmatrix}\,d\sigma_w
\]

\[T^*\begin{pmatrix}
	0\\\phi_2
\end{pmatrix}(z)=-a(\pi)^{-2}\lim_{\epsilon\to 0}\int_{\partial\Omega\backslash\mathbb{B}_\epsilon(z)}|z-w|^{-4}\begin{pmatrix}
	-w_2z_1+w_1z_2\\
-1+w_1\bar{z}_1+w_2\bar{z}_2
\end{pmatrix}\,d\sigma_w\]

	\[R\begin{pmatrix}
	\psi_1\\0
\end{pmatrix}(z)=4b(\pi)^{-2}\underset{\epsilon\to 0}{\finitepart}\bigint_{\partial\Omega\backslash\mathbb{B}_\epsilon(z)}\begin{pmatrix}
	\frac{1}{|z-w|^4}\\
0
\end{pmatrix}\,d\sigma_w.\]

Let $z=(p, 0)$, where $|p|=1$. Observe that \[S\begin{pmatrix}
	0\\\phi_2
\end{pmatrix}(p, 0)=\frac{a}{4\pi^2}\int_{\partial{\Omega}}\frac{1}{|p-w_1|^{2}+|w_2|^2}\begin{pmatrix}
-w_2p\\w_1\bar{p}
\end{pmatrix}\,d\sigma_w.\] One finds that the first component of integrand is odd function on $w_2$ and so the first component  of $S\begin{pmatrix}
p\\0
\end{pmatrix}(p,0)$ is $0$.

\textit{A priori}, for each $b$ and $f$, we may find $a$ so that the $\bar{\partial}$-Neumann problem solves. However, we can observe a rigidity property as follows. It tells us that for $b=0$ and most $f$, there is no appropriate $a$ for a solvable $\bar{\partial}$-Neumann problem.


\begin{theorem}\label{newfound}
	Given an arbitrary $f\in L^2(\mathbb{B}(0,1))$, let $u\in H^2(\Omega)$ solves the $\bar{\partial}$-Neumann problem (\ref{dbarneumann}) with constant boundary velocity. Let $a:=(\B u)_2$ and $b:=(\gamma u)_1$. If $b=0$, then $(\G f, d\bar{z}_2)$ vanishes at $(1,0)$.
	
\end{theorem}
\begin{proof}
	
Assume on the contrary, $(\G f, d\bar{z}_2)\neq 0$ at $(1,0)$. In other words, if we assume $\gamma\G f=sd\bar{z}_1+td\bar{z}_2$, then $t\neq 0$ at $(1,0)$.
Looking at (\ref{boundaryintegral}), the first component of the first equation implies \[\frac{1}{2}\psi_1(z)=(\gamma\G f)_1(z)+(S\phi)_1(z)-\frac{1}{2}(T\psi)_1(z).\] At $z=(1, 0)$, we have that \[\frac{1}{2}\psi_1(1,0)=(\gamma\G f)_1(1,0)-\frac{1}{2}(T\psi)_1(1,0),\] because the discussion before the statement of Theorem \ref{newfound}.  If, moreover, we let $\psi_1=b=0$, one obtains $(\gamma\G f)_1(1,0)=0$.

 To convert $\gamma\G f$ into the coordinate $\omega_{\overline{L}}$ and $\omega_{\overline{N}}$, we simply calculate \[\gamma\G f=(\sqrt{2})^{-1}\begin{pmatrix}
	z_2&-z_1\\\bar{z}_1&\bar{z}_2
\end{pmatrix}\begin{pmatrix}
	s\\t
\end{pmatrix}=(\sqrt{2})^{-1}\begin{pmatrix}
z_2s-z_1t\\\bar{z}_1s+\bar{z}_2t
\end{pmatrix}.\] In other words, $\gamma\G f=(\sqrt{2})^{-1}\left((z_2s-z_1t)\omega_{\overline{L}}+(\bar{z}_1s+\bar{z}_2t)\omega_{\overline{N}}\right)$. At $z=(1,0)$, we have that $\gamma\G f(1,0)=\begin{pmatrix}
-t\\s
\end{pmatrix}$. Since by the assumption $t\neq 0$, we complete the proof by the contradiction with $(\gamma\G f)_1(1,0)=0$.
\end{proof}
\bigskip
\bigskip

	\noindent {\bf Acknowledgments}. The author thanks Dr. Zhaosheng Feng for many fruitful conversations about the Differential Equations and Integral Equations. The author learned a lot from him. A special thank goes to the referee for his/her constructive suggestions.
\printbibliography

\end{document}